\documentclass[11pt,reqno]{amsart}

\usepackage{amssymb}
\usepackage{amsmath,amsthm,amsfonts,amssymb,latexsym,mathrsfs,color,hyperref}
\usepackage{graphicx}
\usepackage{xspace}
\usepackage{multirow}
\usepackage{diagbox}
\usepackage{capt-of}
\usepackage{tikz}
\tikzset{
	level 1/.style = {sibling distance = 1.5cm},
	level 2/.style = {sibling distance = 0.8cm},
    level distance = 1.0 cm
}
\usetikzlibrary {decorations.pathmorphing}
\tikzstyle{snakeline} = [decorate, decoration={snake, amplitude=.4mm, segment length=2mm}]

\newtheorem{theorem}{Theorem}

\newtheorem{proposition}[theorem]{Proposition}

\newtheorem{lemma}[theorem]{Lemma}

\newtheorem{example}[theorem]{Example}
\newtheorem{problem}[theorem]{Problem}

\newcommand{\plat}{{\rm plat\,}}

\newcommand{\des}{{\rm des\,}}

\newcommand{\msn}{\mathfrak{S}_n}

\newcommand{\ms}{\mathfrak{S}}

\newcommand{\lrf}[1]{\lfloor #1\rfloor}

\newcommand{\asc}{{\rm asc\,}}

\title{Commuting Eulerian operators}
\author[S.-M.~Ma]{Shi-Mei Ma}
\address{School of Mathematics and Statistics,
        Northeastern University at Qinhuangdao,
         Hebei 066000, P.R. China}
\email{shimeimapapers@163.com (S.-M. Ma)}
\author[H.~Qi]{Hao Qi}
\address{College of Mathematics and Physics, Wenzhou University, Wenzhou 325035, P.R. China}
\email{qihao@wzu.edu.cn(H.~Qi)}
\author{Jean Yeh}
\address{Department of Mathematics, National Kaohsiung Normal University, Kaohsiung 82444, Taiwan}
\email{chunchenyeh@nknu.edu.tw}
\author[Y.-N. Yeh]{Yeong-Nan Yeh}
\address{College of Mathematics and Physics, Wenzhou University, Wenzhou 325035, P.R. China}
\email{mayeh@math.sinica.edu.tw (Y.-N. Yeh)}
\begin{document}

\maketitle
\begin{abstract}
Motivated by the work of Visontai and Dey-Sivasubramanian on the gamma-positivity of some polynomials, we find the commutative property of
a pair of Eulerian operators. As an application, we
show the bi-gamma-positivity of
the descent polynomials on permutations of the multiset $\{1^{a_1},2^{a_2},\ldots,n^{a_n}\}$, where $0\leqslant a_i\leqslant 2$.
Therefore, these descent polynomials are all alternatingly increasing, and so they are unimodal with modes in the middle.
\bigskip

\end{abstract}
\date{\today}
\section{Introduction}
Let $f(x)=\sum_{i=0}^nf_ix^i$ be a polynomial with nonnegative coefficients.
We say that $f(x)$ is {\it unimodal} if $f_0\leqslant f_1\leqslant \cdots\leqslant f_k\geqslant f_{k+1}\geqslant\cdots \geqslant f_n$ for some $k$,
where the index $k$ is called the {\it mode} of $f(x)$.
It is well known that if $f(x)$ with only nonpositive real zeros, then $f(x)$ is unimodal (see~\cite[p.~419]{Bre94} for instance).
If $f(x)$ is symmetric with the center of symmetry $\lrf{n/2}$, i.e., $f_i=f_{n-i}$ for all indices $0\leqslant i\leqslant n$,
then it can be expanded as
$$f(x)=\sum_{k=0}^{\lrf{{n}/{2}}}\gamma_kx^k(1+x)^{n-2k}.$$
The polynomial $f(x)$ is {\it $\gamma$-positive}
if $\gamma_k\geqslant 0$ for all $0\leqslant k\leqslant \lrf{{n}/{2}}$.
Clearly, $\gamma$-positivity implies symmetry and unimodality.
Let $f(x,y)=\sum_{i=0}^nf_ix^iy^{n-i}$ be a homogeneous bivariate polynomial. We say that $f(x,y)$ is {\it bivariate $\gamma$-positive} with the center of symmetry $\frac{n}{2}$
if $f(x,y)$ can be written as follows:
$$f(x,y)=\sum_{k=0}^{\lrf{{n}/{2}}}\gamma_k(xy)^k(x+y)^{n-2k}.$$
There has been considerable recent interest in the study of the $\gamma$-positivity of polynomials,
see~\cite{Athanasiadis17,Branden14} for details. In particular,
Br\"and\'en~\cite[Remark~7.3.1]{Branden14} noted that
if $f(x)$ is symmetric and has only real zeros, then it is $\gamma$-positive.

Let $f(x)=\sum_{i=0}^nf_ix^i$, where $f_n\neq 0$.
Following~\cite{Beck2010,Branden18}, there is a unique symmetric decomposition $f(x)= a(x)+xb(x)$, where
\begin{equation*}\label{ax-bx-prop01}
a(x)=\frac{f(x)-x^{n+1}f(1/x)}{1-x},~b(x)=\frac{x^nf(1/x)-f(x)}{1-x}.
\end{equation*}
According to~\cite[Definition~8]{Ma2021}, the polynomial $f(x)$ is said to be {\it bi-$\gamma$-positive} if both $a(x)$ and $b(x)$ are $\gamma$-positive.
Thus $\gamma$-positivity is a special case of bi-$\gamma$-positivity.
Following~\cite[Definition 2.9]{Schepers13}, the polynomial $f(x)$ is {\it alternatingly increasing} if
$$f_0\leqslant f_n\leqslant f_1\leqslant f_{n-1}\leqslant\cdots \leqslant f_{\lrf{{(n+1)}/{2}}}.$$
Br\"and\'en and Solus~\cite{Branden18} pointed out that
$f(x)$ is alternatingly increasing if and only if the pair of polynomials in its symmetric decomposition are both unimodal
and have only nonnegative coefficients. Therefore, bi-$\gamma$-positivity is stronger than alternatingly increasing property.
The alternatingly increasing property first appeared in the work of Beck and Stapledon~\cite{Beck2010}.
Recently, Beck-Jochemko-McCullough~\cite{Beck2019}, Br\"and\'en-Solus~\cite{Branden18} and Solus~\cite{Solus19} studied the
alternatingly increasing property of several $h^*$-polynomials as well as some refined Eulerian polynomials.

A multipermutation of a multiset is a sequence of its elements. Throughout this paper, we always let $\mathbf{m}=(m_1,m_2,\ldots,m_n)\in \mathbb{P}^n$.
Denote by $\ms_{\mathbf{m}}$ the set of all multipermutations of
the multiset $\{1^{m_1},2^{m_2},\ldots,n^{m_n}\}$, where $i$ appears $m_i$ times.
Set $m=\sum_{i=1}^nm_i$.
For $\pi=\pi_1\pi_2\ldots\pi_{m}\in\ms_{\mathbf{m}}$, we always assume that $\pi_0=\pi_{m+1}=0$ (except where explicitly stated).
If $i\in \{0,1,2,\ldots,m\}$, then $\pi_{i}$ is called an {\it ascent} (resp.~{\it descent},~{\it plateau}) if $\pi_{i}<\pi_{i+1}$ (resp.~$\pi_{i}>\pi_{i+1}$,~$\pi_{i}=\pi_{i+1}$).
Let $\asc(\pi)$ (resp.~$\des(\pi)$,~$\plat(\pi)$) be the number of ascents (resp.~descents,~plateaux) of $\pi$.
The {\it multiset Eulerian polynomials} $A_{\mathbf{m}}(x)$ are defined by
$$A_{\mathbf{m}}(x)=\sum_{\pi\in\ms_{\mathbf{m}}}x^{\asc(\pi)}=\sum_{\pi\in\ms_{\mathbf{m}}}x^{\des(\pi)}.$$
A classical result of MacMahon~\cite[Vol~2, Chapter IV, p.~211]{MacMahon20} says that
\begin{equation}\label{MacMahon}
\frac{A_{\mathbf{m}}(x)}{(1-x)^{1+m}}
=\sum_{k\geqslant 0}\binom{k+m_1}{m_1}\binom{k+m_2}{m_2}\cdots \binom{k+m_n}{m_n}x^{k+1}.
\end{equation}
Let $\msn$ be the set of all permutations of $\{1,2,\ldots,n\}$. As usual, we write $\pi=\pi_1\pi_2\cdots\pi_n\in\msn$.
Denote by $A_{\pi({\mathbf{m})}}(x)$ the descent polynomial on multipermutations of
$\{\pi_1^{m_1},\pi_2^{m_2},\ldots,\pi_n^{m_n}\}$.
It follows from~\eqref{MacMahon} that
\begin{equation}\label{MacMahon02}
A_{\mathbf{m}}(x)=A_{\pi({\mathbf{m})}}(x).
\end{equation}
When $\mathbf{m}=(1,1,\ldots,1)$, the polynomial $A_{\mathbf{m}}(x)$ is reduced to the classical Eulerian polynomial $A_n(x)$.
In other words,
$$A_n(x)=\sum_{\pi\in\msn}x^{\asc(\pi)}=\sum_{\pi\in\msn}x^{\des(\pi)}.$$

Simion~\cite[Section~2]{Simion84} found that
$A_{\mathbf{m}}(x)$ is real-rootedness for any $\mathbf{m}$.
When $\mathbf{m}=(p,p,\ldots,p)$, Carlitz-Hoggatt~\cite{Carlitz78} showed that $A_{\mathbf{m}}(x)$
is symmetric, where $p$ is a given positive integer.
By~\cite[Remark~7.3.1]{Branden14},
an immediate consequence is the following well known result.
\begin{proposition}\label{prop01}
For any $\mathbf{m}$, the multiset Eulerian polynomials $A_{\mathbf{m}}(x)$ are all unimodal.
When $\mathbf{m}=(p,p,\ldots,p)$, the polynomial $A_{\mathbf{m}}(x)$ is $\gamma$-positive, and so its mode is in the middle.
\end{proposition}

Recently, there has been much work on the
descent polynomials of permutations over multisets, see~\cite{Lin2101,Lin21,Lin22,Liu21,Yan2022} for instance.
In particular, Lin-Xu-Zhao~\cite{Lin22} found a combinatorial interpretation for the
$\gamma$-coefficients of $A_{\mathbf{m}}(x)$ via the model of weakly increasing trees, where $\mathbf{m}=(p,p,\ldots,p)$.
Motivated by Proposition~\ref{prop01}, it is natural to consider the following problem.
\begin{problem}\label{prob01}
For any $\mathbf{m}$, could we
characterize the location of the mode of $A_{\mathbf{m}}(x)$?
\end{problem}

A bivariate version of the Eulerian polynomial over the symmetric group is given as follows:
$$A_n(x,y)=\sum_{\pi\in\msn}x^{\asc(\pi)}y^{\des(\pi)}.$$
In particular, $A_n(x,1)=A_n(1,x)=A_n(x)$.
Carlitz and Scoville~\cite{Carlitz74} found that
\begin{equation*}\label{CarlitzSco74}
A_{n+1}(x,y)=xy\left(\frac{\partial}{\partial x}+\frac{\partial}{\partial y}\right)A_n(x,y),~A_1(x,y)=xy.
\end{equation*}
Using the following Eulerian operator
\begin{equation}\label{operator01}
T=xy\left(\frac{\partial}{\partial x}+\frac{\partial}{\partial y}\right),
\end{equation}
Foata and Sch\"utzenberger~\cite{Foata70} discovered that
\begin{equation*}\label{Anx-gamma}
A_n(x,y)=\sum_{k=1}^{\lrf{({n+1})/{2}}}\gamma(n,k)(xy)^k(x+y)^{n+1-2k},
\end{equation*}
where $\gamma(n,k)$ are all nonnegative integers.
Applying the same idea,
Visontai~\cite{Visontai} investigated the joint generating polynomial of
descents and inverse descents, Dey-Sivasubramanian~\cite{Dey20} studied the descent polynomials on permutations in the alternating group.
As an illustration, we now recall a result on the Eulerian operator $T$, which is a slightly variant of~{\cite[Lemma~5]{Dey20}}.
\begin{lemma}\label{lemma01}
Let $f(x,y)$ be a bivariate $\gamma$-positive polynomial with the center of symmetry $\frac{n}{2}$.
Then $T(f(x,y))$ is a bivariate $\gamma$-positive polynomial with the center of symmetry $\frac{n+1}{2}$.
\end{lemma}

Motivated by the work of Visontai~\cite{Visontai} and Dey-Sivasubramanian~\cite{Dey20},
in this paper we introduce the following Eulerian operator
\begin{equation}\label{operator02}
G=xy^2\left(\frac{\partial}{\partial x}+\frac{\partial}{\partial y}\right)+\frac{x^2y^2}{2}\left(\frac{\partial^2}{\partial x^2}+\frac{\partial^2}{\partial y^2 }\right)+x^2y^2\frac{\partial^2}{\partial x\partial y}.
\end{equation}
In the next section, we prove the commutative property of the Eulerian operators $T$ and $G$. In Section~\ref{Section:03}, we
prove following result, which gives a partial answer to Problem~\ref{prob01}.
\begin{theorem}\label{mainthm1}
Let $\mathbf{m}=\{m_1,m_2,\ldots,m_n\}$, where $0\leqslant m_i\leqslant 2$.
The Eulerian polynomials $A_{\mathbf{m}}(x)$ are all bi-$\gamma$-positive, and so $A_{\mathbf{m}}(x)$ are all alternating increasing.
More precisely, when $\mathbf{m}=\{1,1,\ldots,1\}$ or $\mathbf{m}=\{2,2,\ldots,2\}$, the polynomial $A_{\mathbf{m}}(x)$ is $\gamma$-positive;
for the other cases, the polynomial $A_{\mathbf{m}}(x)$ can be written as a sum of two $\gamma$-positive polynomials.
\end{theorem}
In the following discussion, we always set $\mathbf{m}=\{m_1,m_2,\ldots,m_n\}$, where $0\leqslant m_i\leqslant 2$.
Let $$A_{\mathbf{m}}(x,y)=\sum_{\pi\in\ms_{\mathbf{m}}}x^{\des(\pi)}y^{m+1-\des(\pi)}.$$
where $m=\sum_{i=1}^nm_i$. Clearly, $A_{\mathbf{m}}(x,1)=A_{\mathbf{m}}(x)$. For convenience, set $A_{{\emptyset}}(x,y)=x$.
\begin{example}\label{example6}
We have $$A_{\{1\}}(x,y)=xy,~A_{\{2\}}(x,y)=xy^2,~A_{\{1,1\}}(x,y)=xy(x+y),$$
$$A_{\{1,1,1\}}(x,y)=xy(x^2+4xy+y^2),~A_{\{1,2\}}(x,y)=A_{\{2,1\}}(x,y)=xy^2(y+2x),$$
$$A_{\{2,2\}}(x,y)=xy^2(y^2+4xy+x^2),~A_{\{2,1,2\}}(x,y)=xy^2(y^3+12xy^2+15x^2y+2x^3).$$
\end{example}
\section{The commutative property of Eulerian operators}\label{Section:02}
\begin{lemma}\label{lemma1}
Let $\mathbf{m}=\{m_1,m_2,\ldots,m_n\}$, where $0\leqslant m_i\leqslant 2$.
Set $\overline{\mathbf{m}}=\mathbf{m}\cup\{n+1\}$ and $\underline{\mathbf{m}}=\mathbf{m}\cup\{n+1,n+1\}$.
Let $T$ and $G$ be the Eulerian operators defined by~\eqref{operator01} and~\eqref{operator02}, respectively.
Then we have $A_{\overline{\mathbf{m}}}(x,y)=T\left(A_{\mathbf{m}}(x,y)\right)$ and $A_{\underline{\mathbf{m}}}(x,y)=G\left(A_{\mathbf{m}}(x,y)\right)$.
\end{lemma}
\begin{proof}
Let $\pi\in\ms_{\mathbf{m}}$. We introduce a labeling of $\pi$ as follows:
\begin{itemize}
  \item [\rm ($L_1$)]if $\pi_i$ is a descent, then put a superscript label $x$ right after it;
 \item [\rm ($L_2$)] if $\pi_i$ is an ascent or a plateau, then put a superscript label $y$ right after it.
\end{itemize}
For example, for $\pi=12125433$,
the labeling of $\pi$ is given by $^y1^y2^x1^y2^y5^x4^x3^y3^y$.

When we insert the letter $n+1$ into $\pi$, we always get a label $x$ just before $n+1$ as well as a label $y$ right after $n+1$.
This corresponds to the substitution rule of labels: $x\rightarrow xy$ or $y\rightarrow xy$.
Thus the term $T\left(A_{\mathbf{m}}(x,y)\right)$ gives the contribution of all $\pi'\in\ms_{\overline{\mathbf{m}}}$ in which the element $n+1$ appears in positions $j$, where $0\leqslant j\leqslant m$.
Therefore, one has $A_{\overline{\mathbf{m}}}(x,y)=T\left(A_{\mathbf{m}}(x,y)\right)$.

When we insert two elements $n+1$ into $\pi$, we distinguish among three distinct cases:
\begin{itemize}
  \item [\rm ($c_1$)] If the pair $(n+1)(n+1)$ is inserted in a position of $\pi$, then the changes of labeling are illustrated as follows:
  $$\cdots\pi_i^x\pi_{i+1}\cdots\rightarrow \cdots\pi_i^y(n+1)^y(n+1)^x\pi_{i+1}\cdots,$$
  $$\cdots\pi_i^y\pi_{i+1}\cdots\rightarrow \cdots\pi_i^y(n+1)^y(n+1)^x\pi_{i+1}\cdots.$$
  This explains the term $xy^2\left(\frac{\partial}{\partial x}+\frac{\partial}{\partial y}\right)$;
  \item [\rm ($c_2$)] If the two $n+1$ are inserted into two different positions with the same label, then the changes of labeling are illustrated as follows:
  $$\cdots\pi_i^x\pi_{i+1}\cdots\pi_j^x\pi_{j+1}\cdots\rightarrow \cdots\pi_i^y(n+1)^x\pi_{i+1}\cdots\pi_j^y(n+1)^x\pi_{j+1}\cdots,$$
    $$\cdots\pi_i^y\pi_{i+1}\cdots\pi_j^y\pi_{j+1}\cdots\rightarrow \cdots\pi_i^y(n+1)^x\pi_{i+1}\cdots\pi_j^y(n+1)^x\pi_{j+1}\cdots.$$
    This explains the term $\frac{x^2y^2}{2}\left(\frac{\partial^2}{\partial x^2}+\frac{\partial^2}{\partial y^2 }\right)$;
 \item [\rm ($c_3$)] If the two $n+1$ are inserted into two different positions with different labels, then the changes of labeling are illustrated as follows:
   $$\cdots\pi_i^x\pi_{i+1}\cdots\pi_j^y\pi_{j+1}\cdots\rightarrow \cdots\pi_i^y(n+1)^x\pi_{i+1}\cdots\pi_j^y(n+1)^x\pi_{j+1}\cdots,$$
    $$\cdots\pi_i^y\pi_{i+1}\cdots\pi_j^x\pi_{j+1}\cdots\rightarrow \cdots\pi_i^y(n+1)^x\pi_{i+1}\cdots\pi_j^y(n+1)^x\pi_{j+1}\cdots.$$
    This explains the term $x^2y^2\frac{\partial^2}{\partial x\partial y}$.
\end{itemize}
Therefore, the action of $G$
on the set of labeled multipermutations in $\ms_{\mathbf{m}}$ gives the set of labeled multipermutations in $\ms_{\underline{\mathbf{m}}}$.
This yields $A_{\underline{\mathbf{m}}}(x,y)=G\left(A_{\mathbf{m}}(x,y)\right)$.
\end{proof}

We can now present the following result.
\begin{theorem}\label{thm001}
The Eulerian operators $T$ and $G$ are commutative, i.e., $TG=GT$.
\end{theorem}
\begin{proof}
Let $G=G_1+G_2+G_3$, where
\begin{equation}\label{G1G2G3}
\begin{aligned}
G_1&=xy^2\left(\frac{\partial}{\partial x}+\frac{\partial}{\partial y}\right),~G_2=\frac{x^2y^2}{2}\left(\frac{\partial^2}{\partial x^2}+\frac{\partial^2}{\partial y^2 }\right),~G_3=x^2y^2\frac{\partial^2}{\partial x\partial y}.
\end{aligned}
\end{equation}
It is easily checked that
\begin{align*}
G_1T&=xy^2\left[(x+y)\left(\frac{\partial}{\partial x}+\frac{\partial}{\partial y}\right)+xy\left(\frac{\partial^2}{\partial x^2}+\frac{\partial^2}{\partial y^2}+2\frac{\partial^2}{\partial x\partial y}\right)\right],\\
G_2T&=\frac{x^2y^2}{2}\left[2y\frac{\partial^2}{\partial x^2}+2x\frac{\partial^2}{\partial y^2}+2(x+y)\frac{\partial^2}{\partial x\partial y}+xy\left(\frac{\partial^3}{\partial x^3}+\frac{\partial^3}{\partial y^3}+\frac{\partial^3}{\partial x^2\partial y}+\frac{\partial^3}{\partial y^2\partial x}\right)\right],\\
G_3T&=x^2y^2\left[\left(\frac{\partial}{\partial x}+\frac{\partial}{\partial y}\right)+(x+y)\frac{\partial^2}{\partial x\partial y}+x\frac{\partial^2}{\partial x^2}+y\frac{\partial^2}{\partial y^2}+xy\left(\frac{\partial^3}{\partial x^2\partial y}+\frac{\partial^3}{\partial y^2\partial x}\right)\right],\\
TG_1&=xy\left[(2xy+y^2)\left(\frac{\partial}{\partial x}+\frac{\partial}{\partial y}\right)+xy^2\left(\frac{\partial^2}{\partial x^2}+\frac{\partial^2}{\partial y^2}+2\frac{\partial^2}{\partial x\partial y}\right)\right],\\
TG_2&=xy\left[(xy^2+x^2y)\left(\frac{\partial^2}{\partial x^2}+\frac{\partial^2}{\partial y^2}\right)+\frac{x^2y^2}{2}\left(\frac{\partial^3}{\partial x^3}+\frac{\partial^3}{\partial y^3}+\frac{\partial^3}{\partial x^2\partial y}+\frac{\partial^3}{\partial y^2\partial x}\right)\right],\\
TG_3&=xy\left[2(x^2y+xy^2)\frac{\partial^2}{\partial x\partial y}+x^2y^2\left(\frac{\partial^3}{\partial x^2\partial y}+\frac{\partial^3}{\partial y^2\partial x}\right)\right].
\end{align*}
Thus we obtain
\begin{equation}\label{TG}
\begin{aligned}
GT=TG=(xy^3+2x^2y^2)\left(\frac{\partial}{\partial x}+\frac{\partial}{\partial y}\right)+(2x^2y^3+x^3y^2)\left(\frac{\partial^2}{\partial x^2}+\frac{\partial^2}{\partial y^2}\right)+\\
(4x^2y^3+2x^3y^2)\frac{\partial^2}{\partial x\partial y}+\frac{x^3y^3}{2}\left(\frac{\partial^3}{\partial x^3}+\frac{\partial^3}{\partial y^3}\right)+\frac{3x^3y^3}{2}\left(\frac{\partial^3}{\partial x^2\partial y}+\frac{\partial^3}{\partial y^2\partial x}\right).
\end{aligned}
\end{equation}
This completes the proof.
\end{proof}

\begin{example}
Note that $A_{\{2\}}(x,y)=xy^2$. Using~\eqref{TG}, one has $$GT(xy^2)=TG(xy^2)=xy^2(y^3+12xy^2+15x^2y+2x^3)=A_{\{2,1,2\}}(x,y)=A_{\{2,2,1\}}(x,y).$$
\end{example}
\section{The proof of Theorem~\ref{mainthm1}}\label{Section:03}
We claim that the bivariate $\gamma$-expansions of $A_{\mathbf{m}}(x,y)$ has three types:
\begin{equation}\label{eq01}
A_{\mathbf{m}}(x,y)=\sum_{k=1}^{\lrf{(m+1)/2}}a(m,k)(xy)^k(x+y)^{m+1-2k},
\end{equation}
\begin{equation}\label{eq02}
A_{\mathbf{m}}(x,y)=y\sum_{k=1}^{\lrf{m/2}}b(m,k)(xy)^k(x+y)^{m-2k},
\end{equation}
\begin{equation}\label{eq03}
A_{\mathbf{m}}(x,y)=\sum_{k=1}^{\lrf{(m+1)/2}}c(m,k)(xy)^k(x+y)^{m+1-2k}+y\sum_{k=1}^{\lrf{m/2}}d(m,k)(xy)^k(x+y)^{m-2k},
\end{equation}
where the first expansion corresponds to $\mathbf{m}=\{1,1,\ldots,1\}$, the second expansion corresponds to $\mathbf{m}=\{2,2,\ldots,2\}$,
and the last expansion corresponds to the other cases.

As illustrated by Example~\ref{example6}, the claim holds for any $m\leqslant 4$. We proceed by induction.
It suffices to distinguish among three distinct cases:
\begin{itemize}
  \item [\rm ($a_1$)] Consider the case $\mathbf{m}=\{1,1,\ldots,1\}$. Note that $A_{\mathbf{m}}(x,y)$ is
  bivariate $\gamma$-positive with the center of symmetry $\frac{m+1}{2}$.
  By~\eqref{eq01} and Lemma~\ref{lemma1}, we have
  \begin{align*}
 A_{\overline{\mathbf{m}}}(x,y)&=T\left(A_{\mathbf{m}}(x,y)\right)\\
 &=T\left(\sum_{k=1}^{\lrf{(m+1)/2}}a(m,k)(xy)^k(x+y)^{m+1-2k}\right)\\
 &=\sum_{k}a(m,k)[k(xy)^k(x+y)^{m+2-2k}+2(m+1-2k)(xy)^{k+1}(x+y)^{m-2k}],
  \end{align*}
Setting $\widetilde{a}(m,k)=ka(m,k)+2(m+3-2k)a(m,k-1)$, we get
  \begin{equation}\label{eq002}
  A_{\overline{\mathbf{m}}}(x,y)=\sum_{k=1}^{\lrf{(m+2)/2}}\widetilde{a}(m,k)(xy)^k(x+y)^{m+2-2k}.
   \end{equation}
  Thus the $\gamma$-expansion of $A_{\overline{\mathbf{m}}}(x,y)$ belongs to the type~\eqref{eq01}.

Consider the
action of $G$ on the basis element $(xy)^k(x+y)^{m+1-2k}$. We get
   \begin{align*}
G\left((xy)^k(x+y)^{m+1-2k}\right)=G_1\left((xy)^k(x+y)^{m+1-2k}\right)+\left(G_2+G_3\right)\left((xy)^k(x+y)^{m+1-2k}\right),
  \end{align*}
where $G_1,G_2$ and $G_3$ are defined by~\eqref{G1G2G3}. After some calculations, this gives the following:
   \begin{align*}
&G_1\left((xy)^k(x+y)^{m+1-2k}\right)=y\left[k(xy)^k(x+y)^{m+2-2k}+2(m+1-2k)(xy)^{k+1}(x+y)^{m-2k}\right],\\
&\left(G_2+G_3\right)\left((xy)^k(x+y)^{m+1-2k}\right)=\binom{k}{2}(xy)^k(x+y)^{m+3-2k}+k(xy)^{k+1}(x+y)^{m+3-2(k+1)}+\\
&2k(m+1-2k)(xy)^{k+1}(x+y)^{m+3-2(k+1)}+2(m+1-2k)(m-2k)(xy)^{k+2}(x+y)^{m+3-2(k+2)}.
  \end{align*}
  Thus $G_1\left((xy)^k(x+y)^{m+1-2k}\right)$ and $\left(G_2+G_3\right)\left((xy)^k(x+y)^{m+1-2k}\right)$ are both bivariate $\gamma$-positive polynomials with the center of symmetry $\frac{m+2}{2}$ and $\frac{m+3}{2}$, respectively. Therefore, the $\gamma$-expansion of $G\left(A_{\mathbf{m}}(x,y)\right)$ belongs to the type~\eqref{eq03}.
  More precisely, there exist nonnegative real numbers $\widetilde{c}(m,k)$ and $\widetilde{d}(m,k)$ such that
  \begin{equation}\label{eq04}
  \begin{aligned}
  G\left(A_{\mathbf{m}}(x,y)\right)=\sum_{k=1}^{\lrf{(m+1)/2}}\widetilde{c}(m,k)(xy)^k(x+y)^{m+3-2k}+\\
  y\sum_{k=1}^{\lrf{m/2}}\widetilde{d}(m,k)(xy)^k(x+y)^{m+2-2k}.
 \end{aligned}
  \end{equation}
  \item [\rm ($a_2$)] Consider the case $\mathbf{m}=\{2,2,\ldots,2\}$.
    By~\eqref{eq02} and Lemma~\ref{lemma1}, we have
  \begin{align*}
&T\left(A_{\mathbf{m}}(x,y)\right)\\
&=T\left(y\sum_{k=1}^{\lrf{m/2}}b(m,k)(xy)^k(x+y)^{m-2k}\right)\\
 &=\sum_{k=1}^{\lrf{m/2}}b(m,k)(xy)^{k+1}(x+y)^{m-2k}+yT\left(\sum_{k=1}^{\lrf{m/2}}b(m,k)(xy)^k(x+y)^{m-2k}\right)\\
 &=\sum_{i=2}^{\lrf{(m+2)/2}}b(m,i-1)(xy)^{i}(x+y)^{m+2-2k}+\\
 &y\sum_{k}b(m,k)\left[k(xy)^k(x+y)^{m+1-2k}+2(m-2k)(xy)^{k+1}(x+y)^{m-2k-1}\right],\\
 &=\sum_{i=2}^{\lrf{(m+2)/2}}b(m,i-1)(xy)^{i}(x+y)^{m+2-2k}+y\sum_{k=1}^{\lrf{(m+1)/2}}\widetilde{b}(m,k)(xy)^k(x+y)^{m+1-2k},
  \end{align*}
  where $\widetilde{b}(m,k)=kb(m,k)+2(m-2k+2)b(m,k-1)$. Thus the $\gamma$-expansion of $T\left(A_{\mathbf{m}}(x,y)\right)$ belongs to the type~\eqref{eq03}.

Consider the
action of the operator $G$ on the basis element $y(xy)^p(x+y)^{q}$.
After some simplifications, we obtain that $G\left(x^py^{p+1}(x+y)^{q}\right)$ has the following expansion:
\begin{equation*}
y(xy)^p(x+y)^{q-2}\left[\binom{p}{2}(x+y)^4+(1+p)(1+2q)(xy)(x+y)^2+4\binom{q}{2}(xy)^{2}\right],
\end{equation*}
which yields that the $\gamma$-expansion of $G\left(A_{\mathbf{m}}(x,y)\right)$ belongs to the type~\eqref{eq02}.
More precisely, there exist nonnegative real numbers $\widetilde{b}(m,k)$ such that
\begin{equation}\label{eq05}
G\left(y\sum_{k=1}^{\lrf{m/2}}b(m,k)(xy)^k(x+y)^{m-2k}\right)=y\sum_{k=1}^{\lrf{(m+2)/2}}\widetilde{b}(m,k)(xy)^k(x+y)^{m+2-2k}.
\end{equation}
  \item [\rm ($a_3$)] Consider $\mathbf{m}=\{m_1,m_2,\ldots,m_n\}$, where $\#\{m_i\in\mathbf{m}: m_i=1\}=r$ and  $\#\{m_i\in\mathbf{m}: m_i=2\}=s$.
Without loss of generality, assume that $1\leqslant r,s<n$ and $r+s=n$.
  Combining Lemma~\ref{lemma1} and Theorem~\ref{thm001}, we have
$A_{\mathbf{m}}(x,y)=G^s(T^r\left(x\right))$.
 Using~\eqref{eq002}, we see that there exist nonnegative real numbers $a(r,k)$ such that
   \begin{align*}
G^s(T^r\left(x\right))=G^s\left(\sum_{k=1}^{\lrf{{(r+1)}/2}}a(r,k)(xy)^k(x+y)^{r+1-2k}\right).
  \end{align*}
\end{itemize}
Repeatedly using~\eqref{eq04} and~\eqref{eq05}, we deduce that
$$A_{\mathbf{m}}(x,y)=\sum_{k\geqslant 1}c(r+2s,k)(xy)^k(x+y)^{r+2s+1-2k}+y\sum_{k\geqslant 1}d(r+2s,k)(xy)^k(x+y)^{r+2s-2k}.$$
When $y=1$, we arrive at
$$A_{\mathbf{m}}(x)=\sum_{k\geqslant 1}c(r+2s,k)x^k(1+x)^{r+2s+1-2k}+\sum_{k\geqslant 1}d(r+2s,k)x^k(1+x)^{r+2s-2k},$$
as desired. This completes the proof.
\section*{Acknowledgements.}
The first author was supported by the National Natural Science Foundation of China (Grant number 12071063).
The third author was supported by the National Science Council of Taiwan (Grant number: MOST 110-2115-M-017-002-MY2).

\bibliographystyle{amsplain}

\end{document}